\documentclass[12pt]{amsart}

\usepackage{amssymb}
\usepackage{amsthm}
\usepackage{amsmath}

\usepackage[all,dvips]{xy}

\theoremstyle{plain}
\newtheorem{thm}{Theorem}[section]
\theoremstyle{definition}
\newtheorem{defn}[thm]{Definition}
\theoremstyle{plain}
\newtheorem{prop}[thm]{Proposition}
\theoremstyle{definition}
\newtheorem{example}[thm]{Example}
\theoremstyle{remark}
\newtheorem{rem}[thm]{Remark}

\begin{document}

\title{The Point in Weak Semiprojectivity and AANR Compacta}
\author{Terry A. Loring}
\address{Department of Mathematics and Statistics, University of New Mexico,
Albuquerque, NM 87131, USA.}

\keywords{$C^{*}$-algebras, compactum, lifting, approximative retract.}
\subjclass[2000]{46L85}
\urladdr{http://www.math.unm.edu/\textasciitilde{}loring/}

\begin{abstract}
We initiate the study of pointed approximative absolute neighborhood
retracts. Our motivation is to generate examples of $C^{*}$-algebras
that behave in unexpected ways with respect to weak semiprojectivity.
We consider both weak semiprojectivity (WSP) and weak semiprojectivity
with respect to the class of unital $C^{*}$-algebras (WSP1). For
a non-unital $C^{*}$-algebra, these are different properties. 

One example shows a $C^{*}$-algebra $A$ can fail to be WSP while
its unitization $\widetilde{A}$ is WSP. Another example shows WSP1
is not closed under direct sums.

\end{abstract}

\maketitle

\section{Introduction}

The ``with or without a unit'' choice in $C^{*}$-algebras becomes
serious in the context of certain approximation problems for $C^{*}$-algebras.
We find that weak semiprojectivity for $C_{0}(X)$ in the commutative
category does not translate to a standard condition on $X$. Pointed
approximative absolute neighborhood retracts are introduced, PAANR
for short, as these are the $X$ for which $C_{0}(X)$ is weakly semiprojective.
Our main objective is to understand the commutative $C^{*}$-algebras
based on the spaces shown in Figures~\ref{fig:topologist's-sine-curve},
\ref{fig:two-sine-curves-Nasty} and \ref{fig:two-sine-curves-Nice}
and see what this tells us about weak semiprojectivity.

\begin{figure}
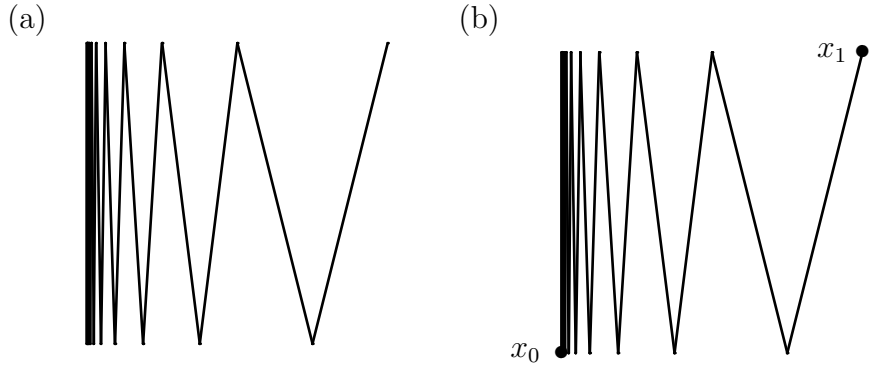

(a)%
\begin{minipage}[t]{0.4\columnwidth}%
\[
\xycompileto{eqn106}{\xygraph{
!~-{@{-}@[|(2.5)]}
!{<0cm,0cm>;<4.0cm,0.0cm>:<0.0cm,4.0cm>::}  
!{(0,1)}*{}="01"
!{(0,0)}*{}="00"
!{(1,1)}*{}="a1"
!{(0.75,0)}*{}="b1"
!{(0.5,1)}*{}="a2"
!{(0.375,0)}*{}="b2"
!{(0.25,1)}*{}="a3"
!{(0.1875,0)}*{}="b3"
!{(0.125,1)}*{}="a4"
!{(0.09375,0)}*{}="b4"
!{(0.0625,1)}*{}="a5"
!{(0.046875,0)}*{}="b5"
!{(0.03125,1)}*{}="a6"
!{(0.0234375,0)}*{}="b6"
!{(0.015625,1)}*{}="a7"
!{(0.01171875,0)}*{}="b7"
!{(0.0078125,1)}*{}="a8"
!{(0.005859375,0)}*{}="b8"
!{(0.00390625,1)}*{}="a9"
!{(0.0029296875,0)}*{}="b9"
!{(0.001953125,1)}*{}="a10"
!{(0.0014648438,0)}*{}="b10"
"00"-"01"
"a1"-"b1"
"b1"-"a2"
"a2"-"b2"
"b2"-"a3"
"a3"-"b3"
"b3"-"a4"
"a4"-"b4"
"b4"-"a5"
"a5"-"b5"
"b5"-"a6"
"a6"-"b6"
"b6"-"a7"
"a7"-"b7"
"b7"-"a8"
"a8"-"b8"
"b8"-"a9"
"a9"-"b9"
"b9"-"a10"
"a10"-"b10"
}}
\]
\end{minipage}\quad{}(b)%
\begin{minipage}[t]{0.4\columnwidth}%
\[
\xycompileto{eqn112}{\xygraph{
!~-{@{-}@[|(2.5)]}
!{<0cm,0cm>;<4.0cm,0.0cm>:<0.0cm,4.0cm>::}  
!{(0,0)}*{\bullet}
!{(-0.12,0)}*{x_0}
!{(1,1)}*{\bullet}
!{(0.90,1)}*{x_1}
!{(0,1)}*{}="01"
!{(0,0)}*{}="00"
!{(1,1)}*{}="a1"
!{(0.75,0)}*{}="b1"
!{(0.5,1)}*{}="a2"
!{(0.375,0)}*{}="b2"
!{(0.25,1)}*{}="a3"
!{(0.1875,0)}*{}="b3"
!{(0.125,1)}*{}="a4"
!{(0.09375,0)}*{}="b4"
!{(0.0625,1)}*{}="a5"
!{(0.046875,0)}*{}="b5"
!{(0.03125,1)}*{}="a6"
!{(0.0234375,0)}*{}="b6"
!{(0.015625,1)}*{}="a7"
!{(0.01171875,0)}*{}="b7"
!{(0.0078125,1)}*{}="a8"
!{(0.005859375,0)}*{}="b8"
!{(0.00390625,1)}*{}="a9"
!{(0.0029296875,0)}*{}="b9"
!{(0.001953125,1)}*{}="a10"
!{(0.0014648438,0)}*{}="b10"
"00"-"01"
"a1"-"b1"
"b1"-"a2"
"a2"-"b2"
"b2"-"a3"
"a3"-"b3"
"b3"-"a4"
"a4"-"b4"
"b4"-"a5"
"a5"-"b5"
"b5"-"a6"
"a6"-"b6"
"b6"-"a7"
"a7"-"b7"
"b7"-"a8"
"a8"-"b8"
"b8"-"a9"
"a9"-"b9"
"b9"-"a10"
"a10"-"b10"
}}
\]
\end{minipage}

\caption{(a): The topologist's sine curve $X$, which is an AANR, even an AAR.
See Theorem~\ref{pro:ARtestForAANR}. (b): This shows points $x_{0}$
and $x_{1}$ considered in various examples. The pointed compacta
$\left(X,x_{0}\right)$ and $\left(X,x_{1}\right)$ behave rather
differently.
\label{fig:topologist's-sine-curve}}
\end{figure}

Let $\alpha X$ be the one-point compactification of the locally compact
metrizable space $X$ with added point $\infty$. We find that
$\left(\alpha X,\infty\right)$ is a PAANR exactly when $C_{0}(X)$
is weakly semiprojective within commutative $C^{*}$-algebras. It is
possible for $\alpha X$ and
$\alpha Y$ to be homeomorphic with $\left(\alpha X,\infty\right)$
not a PAANR while $\left(\alpha Y,\infty\right)$ is a PAANR. This
is not a welcome phenomenon as it implies that weak semiprojectivity
lacks an expected closure property.

Starting from the more civilized topological condition that $\alpha X$
is to be an approximative absolute neighborhood retract (AANR) we
are lead to the study of weak semiprojectivity with respect to unital
$C^{*}$-algebras. This is a condition that applies to unital and
non-unital $C^{*}$-algebras, devolving to weak semiprojectivity in
the unital case. This condition also lacks an expected closure property;
it is not closed under direct sums.

The audience for this note is primarily $C^{*}$-algebraists working
in or near classification or shape theory. The section on PAANR spaces
is hoping to be attractive to a few topologists. That section and
the preceding section that reviews AANR spaces contain no mention
of $C^{*}$-algebras.

To motivate the definition of a PAANR, we take a moment to discuss
two important classes of morphisms between $C^{*}$-algebras and how
these can be constructed from pointed maps and proper maps. Recall
that a \emph{compactum} is a compact, metrizable space. We say
$\left(X,x_{0}\right)$ is a \emph{pointed compactum} when $x_{0}$
is a point in the compact metrizable space $X$. A \emph{pointed map}
$\gamma:\left(X,x_{0}\right)\rightarrow\left(Y,y_{0}\right)$ means
a continuous function $\gamma$ from $X$ to $Y$ such that
$\gamma(x_{0})=\gamma(y_{0})$. To a compactum $X$ we associate the
$C^{*}$-algebra $C(X)$, which is commutative, separable and unital.
Indeed all commutative, separable, unital $C^{*}$-algebras arise
this way, up to isomorphism. If we drop the unital requirement, we
get the slightly broader class of $C^{*}$-algebras of the form
\[
C_{0}\left(X,x_{0}\right)
=
\left\{ f:X\rightarrow\mathbb{C}
\left|\strut\, f\mbox{ is continuous, }f(x_{0})=0
\right.\right\} 
\]
where $\left(X,x_{0}\right)$ varies over the pointed compacta. Alternately,
we can describe these as $C_{0}\left(Y\right)$ where $Y$ is locally
compact and metrizable. But what of the morphisms?

The default choice of morphisms between $C^{*}$-algebras are the
$*$-homomorphisms. Between $C(X)$ and $C(Y)$ the unital
$*$-homomorphisms are of the form $f\mapsto f\circ\gamma$ for
a continuous map $\gamma:Y\rightarrow X$.
We miss out on the non-unital $*$-homomorphisms, but this turns out
to be a trivial matter. For example, when $Y$ is connected, the only
non-unital $*$-homomorphism is zero. But what of the non-unital case?

\textbf{Myth}. The $*$-homomorphisms from $C_{0}\left(X,x_{0}\right)$
to $C_{0}\left(Y,y_{0}\right)$ are all of the form $f\mapsto f\circ\gamma$
for a continuous, proper map $\gamma$ from $Y\setminus\{y_{0}\}$
to $X\setminus\{x_{0}\}$. 

\textbf{Fact}. Only the so-called proper $*$-homomorphisms arise
from proper maps, and there are lots of non-proper $*$-homomorphisms.
See\cite{ELP-pushBusby}.

The $*$-homomorphisms from $C_{0}\left(X,x_{0}\right)$ to
$C_{0}\left(Y,y_{0}\right)$ are all of the form
$f\mapsto f\circ\gamma$ for a pointed map
$\gamma:\left(X,x_{0}\right)\rightarrow\left(Y,y_{0}\right)$.
Recall that pointed requires $\gamma(x_{0})=\{y_{0}\}$ only. If we
also required $\gamma^{-1}(\{y_{0}\})=\{x_{0}\}$ \emph{then} we would
have, by restriction, a proper map from $Y\setminus\{y_{0}\}$ to
$X\setminus\{x_{0}\}$. 

We are focused on examples that show some odd behavior in $C^{*}$-algebras
and in getting down all the details on the relation between weak
semiprojectivity and (P)AANR compacta. The issue of when the AR, ANR, or
AANR property for compacta $X$ is sufficient to make
$C_{0}(X\setminus\{x_{0}\})$ projective, or $C(X)$ semiprojective or
weakly semiprojective, has been much researched lately. See the papers
of Chigogidze and Dranishnikov \cite{chigogidzeNonComARs},
S\o rensen and Thiel \cite{sorensen2011characterization},
and Enders \cite{enders2011characterization}.

All the spaces considered will be one-dimensional.  The reason for this
is that when $X$ is two-dimensional, the $C^*$-algebra $C(X)$ will fail
to be weakly semiprojective.  Indeed, the converse holds
\cite{chigogidzeNonComARs,sorensen2011characterization,enders2011characterization}.
We are considering spaces with potentially very aberant local structure, so
even ignoring interaction with $C^*$-algebras, the topology of a
one-dimensional space can be tricky \cite{cannon2006fundamental}.

\begin{figure}
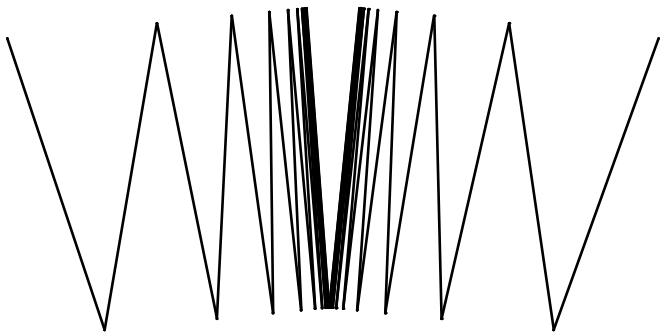

\[
\xycompileto{eqn107}{\xygraph{
!~-{@{-}@[|(2.5)]}
!{<0cm,0cm>;<3.98cm,-0.4cm>:<0.4cm,3.98cm>::}  
!{(0,1)}*{}="01"
!{(0,0)}*{}="00"
!{(1,1)}*{}="a1"
!{(0.75,0)}*{}="b1"
!{(0.5,1)}*{}="a2"
!{(0.375,0)}*{}="b2"
!{(0.25,1)}*{}="a3"
!{(0.1875,0)}*{}="b3"
!{(0.125,1)}*{}="a4"
!{(0.09375,0)}*{}="b4"
!{(0.0625,1)}*{}="a5"
!{(0.046875,0)}*{}="b5"
!{(0.03125,1)}*{}="a6"
!{(0.0234375,0)}*{}="b6"
!{(0.015625,1)}*{}="a7"
!{(0.01171875,0)}*{}="b7"
!{(0.0078125,1)}*{}="a8"
!{(0.005859375,0)}*{}="b8"
!{(0.00390625,1)}*{}="a9"
!{(0.0029296875,0)}*{}="b9"
!{(0.001953125,1)}*{}="a10"
!{(0.0014648438,0)}*{}="b10"
!{<0cm,0cm>;<-3.98cm,-0.4cm>:<-0.3cm,3.98cm>::}   
!{(0,1)}*{}="z01"
!{(0,0)}*{}="z00"
!{(1,1)}*{}="za1"
!{(0.75,0)}*{}="zb1"
!{(0.5,1)}*{}="za2"
!{(0.375,0)}*{}="zb2"
!{(0.25,1)}*{}="za3"
!{(0.1875,0)}*{}="zb3"
!{(0.125,1)}*{}="za4"
!{(0.09375,0)}*{}="zb4"
!{(0.0625,1)}*{}="za5"
!{(0.046875,0)}*{}="zb5"
!{(0.03125,1)}*{}="za6"
!{(0.0234375,0)}*{}="zb6"
!{(0.015625,1)}*{}="za7"
!{(0.01171875,0)}*{}="zb7"
!{(0.0078125,1)}*{}="za8"
!{(0.005859375,0)}*{}="zb8"
!{(0.00390625,1)}*{}="za9"
!{(0.0029296875,0)}*{}="zb9"
!{(0.001953125,1)}*{}="za10"
!{(0.0014648438,0)}*{}="zb10"
"00"-"01" "z00"-"z01" 
"a1"-"b1" "za1"-"zb1"
"b1"-"a2" "zb1"-"za2"
"a2"-"b2" "za2"-"zb2"
"b2"-"a3" "zb2"-"za3"
"a3"-"b3" "za3"-"zb3" 
"b3"-"a4" "zb3"-"za4" 
"a4"-"b4" "za4"-"zb4"
"b4"-"a5" "zb4"-"za5"
"a5"-"b5" "za5"-"zb5"
"b5"-"a6" "zb5"-"za6"
"a6"-"b6" "za6"-"zb6"
"b6"-"a7" "zb6"-"za7"
"a7"-"b7" "za7"-"zb7"
"b7"-"a8" "zb7"-"za8"
"a8"-"b8" "za8"-"zb8"
"b8"-"a9" "zb8"-"za9"
"a9"-"b9" "za9"-"zb9"
"b9"-"a10" "zb9"-"za10"
"a10"-"b10" "za10"-"zb10"
}}
\]\caption{Two copies of the topologist's sine curve, attached at a point so
that the result is not an AANR. See Theorem~\ref{pro:AANRviaApproxSplit}.
\label{fig:two-sine-curves-Nasty}}
\end{figure}

\begin{figure}
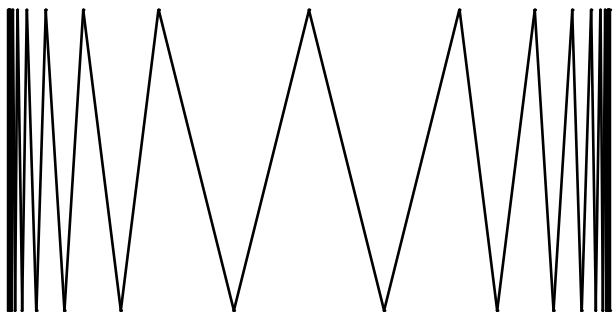

\[
\xycompileto{eqn108}{\xygraph{
!~-{@{-}@[|(2.5)]}
!{<0cm,0cm>;<4.00cm,-0cm>:<0cm,4.00cm>::}  
!{(0,1)}*{}="01"
!{(0,0)}*{}="00"
!{(1,1)}*{}="a1"
!{(0.75,0)}*{}="b1"
!{(0.5,1)}*{}="a2"
!{(0.375,0)}*{}="b2"
!{(0.25,1)}*{}="a3"
!{(0.1875,0)}*{}="b3"
!{(0.125,1)}*{}="a4"
!{(0.09375,0)}*{}="b4"
!{(0.0625,1)}*{}="a5"
!{(0.046875,0)}*{}="b5"
!{(0.03125,1)}*{}="a6"
!{(0.0234375,0)}*{}="b6"
!{(0.015625,1)}*{}="a7"
!{(0.01171875,0)}*{}="b7"
!{(0.0078125,1)}*{}="a8"
!{(0.005859375,0)}*{}="b8"
!{(0.00390625,1)}*{}="a9"
!{(0.0029296875,0)}*{}="b9"
!{(0.001953125,1)}*{}="a10"
!{(0.0014648438,0)}*{}="b10"
!{<8.00cm,0cm>+<0.0cm,0cm>;<8.00cm,0cm>+<-4.00cm,-0cm>:<8.00cm,0cm>+<-0cm,4.00cm>::}   
!{(0,1)}*{}="z01"
!{(0,0)}*{}="z00"
!{(1,1)}*{}="za1"
!{(0.75,0)}*{}="zb1"
!{(0.5,1)}*{}="za2"
!{(0.375,0)}*{}="zb2"
!{(0.25,1)}*{}="za3"
!{(0.1875,0)}*{}="zb3"
!{(0.125,1)}*{}="za4"
!{(0.09375,0)}*{}="zb4"
!{(0.0625,1)}*{}="za5"
!{(0.046875,0)}*{}="zb5"
!{(0.03125,1)}*{}="za6"
!{(0.0234375,0)}*{}="zb6"
!{(0.015625,1)}*{}="za7"
!{(0.01171875,0)}*{}="zb7"
!{(0.0078125,1)}*{}="za8"
!{(0.005859375,0)}*{}="zb8"
!{(0.00390625,1)}*{}="za9"
!{(0.0029296875,0)}*{}="zb9"
!{(0.001953125,1)}*{}="za10"
!{(0.0014648438,0)}*{}="zb10"
"00"-"01" "z00"-"z01" 
"a1"-"b1" "za1"-"zb1"
"b1"-"a2" "zb1"-"za2"
"a2"-"b2" "za2"-"zb2"
"b2"-"a3" "zb2"-"za3"
"a3"-"b3" "za3"-"zb3" 
"b3"-"a4" "zb3"-"za4" 
"a4"-"b4" "za4"-"zb4"
"b4"-"a5" "zb4"-"za5"
"a5"-"b5" "za5"-"zb5"
"b5"-"a6" "zb5"-"za6"
"a6"-"b6" "za6"-"zb6"
"b6"-"a7" "zb6"-"za7"
"a7"-"b7" "za7"-"zb7"
"b7"-"a8" "zb7"-"za8"
"a8"-"b8" "za8"-"zb8"
"b8"-"a9" "zb8"-"za9"
"a9"-"b9" "za9"-"zb9"
"b9"-"a10" "zb9"-"za10"
"a10"-"b10" "za10"-"zb10"
}}
\]\caption{Two copies of the topologist's sine curve, attached at
a point so that the result is an AANR. See
Theorem~\ref{pro:ARtestForAANR}.
\label{fig:two-sine-curves-Nice}}
\end{figure}

Whenever looking at forms of semiprojectivity, a strong motivation
is potential applications in shape theory. For a survey of shape theory
of spaces, with even some remarks about shape theory for $C^{*}$-algebras,
see \cite{mardesic2001history}. For a treatment of shape theory for
$C^{*}$-algebras that connects it with $E$ -theory, see
\cite{DadarlatShapeAndE},

The definition of PAANR, and some basic results, circulated in an
early version of \cite{LoringWeaklyProjective}. In final form, that
paper focused on a new property for $C^{*}$-algebras, being weakly
projective, and the related topological concept of being a pointed
approximative absolute retract (PAAR). Weak semiprojectivity has
been studied for some time
\cite{ELP-anticommutation,HadwinLiApproxLift,SpielbergWeakSemiprojectivity},
having been introduced at least as early as 1997 in
\cite{Loring-lifting-perturbing}. 

The author thanks Adam S\o rensen and Hannes Thiel for feedback on
the exposition of this work and many discussions related to shape
theory.

\section{AANR spaces}
\label{sec:AANR-spaces}

We start with a careful review of approximative absolute neighborhood
retract (AANR) in the sense of \cite{ClappAANR}. We are especially
interested in an equivalent formulation, following ideas from
\cite{Blackadar-shape-theory}, that translates to a lifting problem
in $C^{*}$-algebras. 

\begin{defn}
\label{def:AANR}
A compactum $X$ is an \emph{approximative absolute
neighborhood retract (AANR)} if, for every homeomorphic embedding
$\theta:X\rightarrow Y$ of $X$ into a compact metric space $(Y,d)$,
and for every $\epsilon>0$, there exists $\delta>0$ and a continuous
function $r:U_{\delta}\rightarrow X$ so that
\[
d(r\circ\theta(x),x)\leq\epsilon
\]
for all $x$ in $X,$ where 
\begin{equation}
U_{\delta}=\left\{ y\in Y\left|\, d(y,\theta(X))\leq\delta\right.\right\} .
\label{eq:def-of-U_delta}
\end{equation}

\end{defn}
Clapp asks only that $r$ be defined on a neighborhood $N(\epsilon)$
of $\theta(X)$, but this is equivalent since every open set containing
the compact set $\theta(X)$ contains some $U_{\delta},$ and each
$U_{\delta}$ contains the open neighborhood
\[
\left\{ y\in Y\left|\, d(y,\theta(X))<\delta\right.\right\} .
\]
We can gain flexibility in applying the AANR property by allowing
for more general decreasing sets. We also downplay the metric on $X$
insisting only that we have uniform convergence. It is key here that
$X$ be compact so that we have uniform equivalence of any two compatible
metrics and so the uniform convergence in (\ref{eq:unif_to_id}) does
not depend on the choice of metric on $X$.

\begin{prop}
\label{pro:AANRviaApproxSplit} 
A compactum $X$ is an AANR if, and
only if, for every continuous embedding $\theta:X\rightarrow Y$ of
$X$ into a compactum $Y$, and for every sequence 
$Y_{1}\supseteq Y_{2}\supseteq\cdots$ of closed subsets of $Y$
with $\bigcap Y_{n}=\theta(X),$ there exists a sequence of
continuous functions $r_{n}:Y_{n}\rightarrow X$ so that
\begin{equation}
\lim_{n\rightarrow\infty}r_{n}(\theta(x))=x
\label{eq:unif_to_id}
\end{equation}
uniformly over $x$ in $X.$ 
\end{prop}

\begin{proof}
Let $d$ be a compatible metric on $Y.$ Then $U_{\frac{1}{n}}$ is
a decreasing sequence of closed subsets with intersection $\theta(X)$
and so this condition easily implies $X$ is an AANR. 

Now assume $X$ is an AANR and that the $Y_{n}$ are given. For a
given $k>0$ we know there is a $\delta>0$ and continuous map 
$r:U_{\delta}\rightarrow X$ so that
\[
d(r(\theta(x)),x)\leq\frac{1}{k}
\]
for all $x$ in $X.$ Since $\theta(X)$ and $Y_{n}$ are compact,
we will see that for some $n_{k}$ we have the inclusion
$Y_{n_{k}}\subseteq U_{\frac{1}{k}}$ and so can define $r_{n_{k}}$
as the restriction of $r$ to $Y_{n_{k}}.$  If there is no such
inclusion, then we have $y_{1},y_{2},\ldots$ with $y_{n}\in Y_{n}$
and $d(y_{n},\theta(x))\geq\frac{1}{k}$ for all $x\in X$. Passing
to a subsequence we have $y=\lim_{n}y_{n}$ in $\theta(X)$ with
$d(y,\theta(x))\geq\frac{1}{k}$ for all $x\in X,$ and so
$d(y,y)>\frac{1}{k}$, a contradiction. We can arrange that the
$n_{k}$ are increasing, as define $r_{\ell}:Y_{\ell}\rightarrow X$
as the restriction of $r_{n_{k}}$ whenever $k$ is between $n_{k}$
and $n_{k+1}$. Finally, we define the initial $r_{1},\dots,r_{n_{1}-1}$
in any way we like.
\end{proof}

We get an even more useful characterization, an approximate local
extension property. 

\begin{prop}
\label{pro:AANRviaApproxExten}
A compactum $X$ is an AANR if, and
only if, for every closed subset $Y$ of a compact metrizable space
$Z,$ for every sequence $Y_{1}\supseteq Y_{2}\supseteq\cdots$ of
closed subsets of $Z$ with $\bigcap Y_{n}=Y,$ and for every continuous
function $\lambda:Y\rightarrow X,$ there is a sequence of continuous
functions $\lambda_{n}:Y_{n}\rightarrow X$ so that 
\[
\lim_{n\rightarrow\infty}\lambda_{n}(y)=\lambda(y)
\]
uniformly over $y$ in $Y.$ To summarize in a diagram: 
\begin{equation}
\xycompileto{eqn101}{\xymatrix{
 &	Z\\
 & Y_n \ar@{-->}[dl] _{ \lambda_n }  \ar@{^{(}->}[u] \\
 X & Y \ar[l] ^{\lambda} \ar@{^{(}->}[u]  
}
}\label{eqn:AANRextensionDiagram}
\end{equation}
\end{prop}

\begin{proof}
Suppose $X$ is an AANR and we are given $Y,$ $Z,$ $\lambda$ as indicated,
\[
\xycompileto{eqn102}{\xymatrix{
 & Z \\
 X & Y \ar[l] ^{\lambda} \ar@{^{(}->}[u] 
}}
\]
and that there are closed subsets to
that $\bigcap Y_{n}=Y$. Take the pushout,
\[
\xycompileto{eqn103}{\xymatrix{
 X\cup_{Y}Z  & Z \ar[l] _(0.3){ \iota_Z }  \\
 X     \ar[u] ^ { \iota_X }     & Y  \ar[l] ^{ \lambda } \ar@{^{(}->}[u]
}
}
\]
in which $X\cup_{Y}Z$ is a compact metrizable space and $\iota_{X}$
is one-to-one, continuous, and so a homeomorphism onto its image.
We can select a compatible metric on $X\cup_{Y}Z$ and then the
corresponding metric on $X$ to make $\iota_{X}$ an isometry. 

Consider the closed sets
\[
\iota_{X}(X)\cup\iota_{Z}\left(Y_{n}\right)
\]
that have intersection
\[
\iota_{X}(X)\cup\iota_{Z}\left(Y\right)
=
\iota_{X}(X)\cup\iota_{X}\left(\lambda(Y)\right)
=
\iota_{X}(X).
\]
Applying Proposition~\ref{pro:AANRviaApproxSplit} we find maps 
\[
\rho_{n}:\iota_{X}(X)\cup\iota_{Z}\left(Y_{n}\right)\rightarrow X
\]
 so that
 \[
\lim_{n\rightarrow\infty}\rho_{n}\circ\iota_{X}(x)=x
\]
uniformly over $x$ in $X$. We then let $r_{n}$ be defined on $Y_{n}$
by
\[
r_{n}(y)=\rho_{n}(\iota_{Z}(y))
\]
so that when $y$ is in $Y$ we have 
\[
\lim_{n\rightarrow\infty}r_{n}(y)
=
\lim_{n\rightarrow\infty}\rho_{n}(\iota_{Z}(y))
=
\lim_{n\rightarrow\infty}\rho_{n}\circ\iota_{X}(\lambda(y))
=
\lambda(y)
\]
and this convergence in uniform simply because the convergence
$\rho_{n}\circ\iota_{X}\rightarrow\mathrm{id}$ is uniform.

For the other implication we will use
Proposition~\ref{pro:AANRviaApproxSplit}. Suppose we are given
$\theta:X\rightarrow Y$ with $Y_{n}$ decreasing closed sets such
that $\bigcap_{n}Y_{n}=Y$. Here $\theta$ is assumed to be a
continuous embedding, but we can go further and select compatible
metrics on $X$ and $Y$ so that $\theta$ is an isometry. We apply
the assumed condition to $\theta^{-1}:\theta(X)\rightarrow X$ and
so find continuous functions $r_{n}:Y_{n}\rightarrow X$ with
\[
\lim_{n\rightarrow\infty}r_{n}(y)=\theta^{-1}(y)
\]
uniformly over $y$ in $\theta(Y)$. As $\theta$ is an isometry,
this is equivalent to 
\[
\lim_{n\rightarrow\infty}r_{n}(\theta(x))=x
\]
uniformly over $x$ in $X$. 
\end{proof}

Now we head the other way, looking for a very restrictive approximate
retraction problem that will be useful for showing a space is an AANR.
We use it when, as is so often the case, $X$ is given to us as a compact
subset in Euclidean space and so sits inside a hypercube, or in some
other absolute retract (AR).

\begin{prop}
\label{pro:ARtestForAANR}
Suppose $X$ is a closed subset of $Q$
where $Q$ is an absolute retract. Suppose $X_{1}\supseteq X_{2}\supseteq\cdots$
are closed subset with $\bigcap X_{n}=X$ and where for each $n$
the interior of $X_{n}$ contains $X$. Then $X$ is an AANR if, and
only if, there is a sequence of continuous functions
$r_{n}:X_{n}\rightarrow X$ so that 
\[
\lim_{n\rightarrow\infty}r_{n}(x)=x
\]
uniformly over $x$ in $X$. 
\end{prop}

\begin{proof}
The only nontrivial implication is the backwards one.

Suppose we are given a continuous embedding $\theta:X\rightarrow Y$
into a compactum $Y$ and a sequence
$Y_{1}\supseteq Y_{2}\supseteq\cdots$ of closed subsets of $Y$
with $\bigcap Y_{n}=\theta(X)$. Consider
the diagram
\[
\xycompileto{eqn104}{\xymatrix{
Q
	& Y 
		\\
X_n	\ar@{_{(}->}[u]
	& Y_k \ar@{_{(}->}[u]
		\\
X	\ar[r] _{\theta}	\ar@{_{(}->}[u]
	&  \theta(X) \ar@{_{(}->}[u]
}
}  
\]
We apply the extension property of
AR spaces to the map $\theta^{-1}$ to get $\alpha:Y\rightarrow Q$
so that $\alpha(\theta(x))=x$ for all $x$ in $X$. We also have
the assumed $r_{n}$ which we indicate now as well,
\[
\xycompileto{eqn105}{\xymatrix{
Q
	& Y \ar[l] _{\alpha}
		\\
X_n	\ar@{_{(}->}[u] \ar@/_/[d] _{r_n}
	& Y_k \ar@{_{(}->}[u]
		\\
X	\ar[r] _{\theta}	\ar@{_{(}->}[u]
	&  \theta(X) \ar@{_{(}->}[u]
}
}  
\]
where the diagram is commutative except as it involves $r_{n}$, where
we have $r_{n}(x)\rightarrow x$ uniformly over $x$ in $X$.

Next we calculate the intersection of the $\alpha(Y_{k})$. Easily
we see
\[
\bigcap_{k}\alpha(Y_{k})\supseteq\alpha\left(\bigcap_{k}Y_{k}\right)
=
\alpha\left(\theta(X)\right)
=
X
\]
and so suppose $p$ is in $\bigcap_{k}\alpha(Y_{k})$. This means
$p=\alpha(y_{k})$ for $y_{k}$ in $Y_{k}$. The ambient space $Y$
is compact, so we can pass to a subsequence $y_{k_{\ell}}$ so that
the limit $y$ exists. Notice $y$ must be in $\theta(X)$ by the
assumptions on the $Y_{k}$ and so
\[
p
=
\lim_{\ell}p=\lim_{\ell}\alpha\left(y_{k_{\ell}}\right)
=
\alpha\left(\lim_{\ell}y_{k_{\ell}}\right)
=
\alpha(y)\in\alpha\left(\theta(X)\right)
=
X,
\]
establishing the expected equality 
\[
\bigcap_{k}\alpha(Y_{k})=X.
\]

Fix $n$. The sets $\alpha(Y_{k})$ are compact and decreasing, so
are eventually contained the interior of $X_{n}$. We can find a subsequence
$Y_{k_{1}},Y_{k_{2}},\dots$ so that 
$\alpha\left(Y_{k_{n}}\right)\subseteq X_{n}$.
We define $\rho_{k_{n}}:Y_{k_{n}}\rightarrow X$
by
\[
\rho_{k_{n}}(y)=r_{n}(\alpha(y)).
\]
We define $\rho_{1}$ though $\rho_{k_{1}-1}$ at random and for
$\ell$ strictly between $k_{n}$ and $k_{n+1}$ we define
$\rho_{\ell}$ to be the restriction of $\rho_{k_{n}}$ to
$Y_{\ell}$ to ensure $\rho_{\ell}\left(\theta(x)\right)$ is
the same as $\rho_{k_{n}}\left(\theta(x)\right)$ except with
each term possibly repeated. This will have no effect on the
uniformity of the convergence. For any $x$ in $X$ we find
\begin{align*}
\lim_{\ell\rightarrow\infty}\rho_{\ell}\left(\theta(x)\right)
& =\lim_{n\rightarrow\infty}\rho_{k_{n}}\left(\theta(x)\right)\\
& =\lim_{n\rightarrow\infty}r_{n}(\alpha(\theta(x)))\\
& =\alpha\left(\theta(x)\right)\\
& =x
\end{align*}
and the convergence in uniform because $r_{n}\rightarrow\mathrm{id}$
uniformly and $\theta$ must be uniformly continuous. 
\end{proof}

\begin{figure}
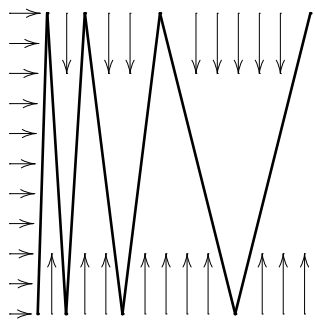

\[
\xycompileto{eqn109}{\xygraph{
!~-{@{-}@[|(2.5)]}
!{<0cm,0cm>;<4.0cm,0.0cm>:<0.0cm,4.0cm>::}  
!{(0,1)}*{}="01"
!{(0,0)}*{}="00"
!{(1,1)}*{}="a1"
!{(0.75,0)}*{}="b1"
!{(0.5,1)}*{}="a2"
!{(0.375,0)}*{}="b2"
!{(0.25,1)}*{}="a3"
!{(0.1875,0)}*{}="b3"
!{(0.125,1)}*{}="a4"
!{(0.09375,0)}*{}="b4"
!{(0.0625,1)}*{}="a5"
!{(0.046875,0)}*{}="b5"
!{(0.03125,1)}*{}="a6"
!{(0.0234375,0)}*{}="b6"
!{(0.015625,1)}*{}="a7"
!{(0.01171875,0)}*{}="b7"
!{(0.0078125,1)}*{}="a8"
!{(0.005859375,0)}*{}="b8"
!{(0.00390625,1)}*{}="a9"
!{(0.0029296875,0)}*{}="b9"
!{(0.001953125,1)}*{}="a10"
!{(0.0014648438,0)}*{}="b10"
!{(0.00,0.0)}*{}="L0"
!{(0.00,0.1)}*{}="L1"
!{(0.00,0.2)}*{}="L2"
!{(0.00,0.3)}*{}="L3"
!{(0.00,0.4)}*{}="L4"
!{(0.00,0.5)}*{}="L5"
!{(0.00,0.6)}*{}="L6"
!{(0.00,0.7)}*{}="L7"
!{(0.00,0.8)}*{}="L8"
!{(0.00,0.9)}*{}="L9"
!{(0.00,1.0)}*{}="LX"
!{(0.07,0.0)}*{}="M0"
!{(0.0727,0.1)}*{}="M1"
!{(0.0754,0.2)}*{}="M2"
!{(0.0781,0.3)}*{}="M3"
!{(0.0808,0.4)}*{}="M4"
!{(0.0835,0.5)}*{}="M5"
!{(0.0862,0.6)}*{}="M6"
!{(0.0889,0.7)}*{}="M7"
!{(0.0916,0.8)}*{}="M8"
!{(0.0943,0.9)}*{}="M9"
!{(0.10,1.0)}*{}="MX"
!{(0.14,0.0)}*{}="U0"
!{(0.14,0.2)}*{}="N0"
!{(0.25,0.0)}*{}="U1"
!{(0.25,0.2)}*{}="N1"
!{(0.32,0.0)}*{}="U2"
!{(0.32,0.2)}*{}="N2"
!{(0.45,0.0)}*{}="U3"
!{(0.45,0.2)}*{}="N3"
!{(0.52,0.0)}*{}="U4"
!{(0.52,0.2)}*{}="N4"
!{(0.59,0.0)}*{}="U5"
!{(0.59,0.2)}*{}="N5"
!{(0.66,0.0)}*{}="U6"
!{(0.66,0.2)}*{}="N6"
!{(0.84,0.0)}*{}="U7"
!{(0.84,0.2)}*{}="N7"
!{(0.91,0.0)}*{}="U8"
!{(0.91,0.2)}*{}="N8"
!{(0.98,0.0)}*{}="U9"
!{(0.98,0.2)}*{}="N9"
!{(0.19,1.0)}*{}="B0"
!{(0.19,0.8)}*{}="K0"
!{(0.33,1.0)}*{}="B1"
!{(0.33,0.8)}*{}="K1"
!{(0.40,1.0)}*{}="B2"
!{(0.40,0.8)}*{}="K2"
!{(0.62,1.0)}*{}="B3"
!{(0.62,0.8)}*{}="K3"
!{(0.69,1.0)}*{}="B4"
!{(0.69,0.8)}*{}="K4"
!{(0.76,1.0)}*{}="B5"
!{(0.76,0.8)}*{}="K5"
!{(0.83,1.0)}*{}="B6"
!{(0.83,0.8)}*{}="K6"
!{(0.90,1.0)}*{}="B7"
!{(0.90,0.8)}*{}="K7"
"B0":"K0"
"B1":"K1"
"B2":"K2"
"B3":"K3"
"B4":"K4"
"B5":"K5"
"B6":"K6"
"B7":"K7"
"U0":"N0"
"U1":"N1"
"U2":"N2"
"U3":"N3"
"U4":"N4"
"U5":"N5"
"U6":"N6"
"U7":"N7"
"U8":"N8"
"U9":"N9"
"L0":"M0"
"L1":"M1"
"L2":"M2"
"L3":"M3"
"L4":"M4"
"L5":"M5"
"L6":"M6"
"L7":"M7"
"L8":"M8"
"L9":"M9"
"LX":"MX"
"a1"-"b1"
"b1"-"a2"
"a2"-"b2"
"b2"-"a3"
"a3"-"b3"
"b3"-"a4"
"a4"-"b4"
}}
\]
\caption{An illustration of the approximate retraction of the unit square onto
the topologist's sine curve from Figure~\ref{fig:topologist's-sine-curve}.
\label{fig:topologist's-sine-curve-1}}
\end{figure}

\begin{example}
\label{exa:standard-example} 
A standard example of an AANR is the
topologist's sine curve $X$ as illustrated in 
Figure~\ref{fig:topologist's-sine-curve}.  This is moreover an AAR,
meaning an approximative absolute retract, as was observed by
Clapp \cite{ClappAANR}. The essential argument here is that the
square in which $X$ is embedded can be mapped to $X$ so as to fix
the points of $X$ except for those in a small region on the left
of the square. These are to be mapped a little horizontally to
a segment in $X$. The rest of is mapped vertically to $X$. This
approximate retraction is illustrated in
Figure~\ref{fig:topologist's-sine-curve-1}. 
\end{example}

\begin{example}
\label{exa:JoinNotAANR} 
Joining two copies of the topologist's sine curve at a point, as
indicated in Figure~\ref{fig:topologist's-sine-curve}, leads to
a space $X$ that is not an AANR. Consider the closed neighborhoods
$X_{n}$ of $X$ as illustrated in
Figure~\ref{fig:two-sine-curves-Nasty-why}.  Each $X_{n}$ consists
of a V-shaped bar in the center and a thin strip around the zig-zag
away from the center. These $X_{n}$ are all path connected, and
$X=\bigcap_{n}X_{n}$. Were $X$ an AANR then
Proposition~\ref{pro:AANRviaApproxSplit} would give us maps
$r_{n}:X_{n}\rightarrow X$ that move points in $X$ by no more that
a given $\epsilon$. Since $X$ has three path-components, $r_{n}(X)$
must lie entirely in one of these path-components, so in the left
zig-zag of $X$, the right zig-zag of $X$, or the middle V-shape of
$X$, all of which have diameter less than the diameter of $X$.
This contradicts the fact that the two outermost points of $X$
are moved very little by $r_{n}$.
\end{example}

\begin{figure}
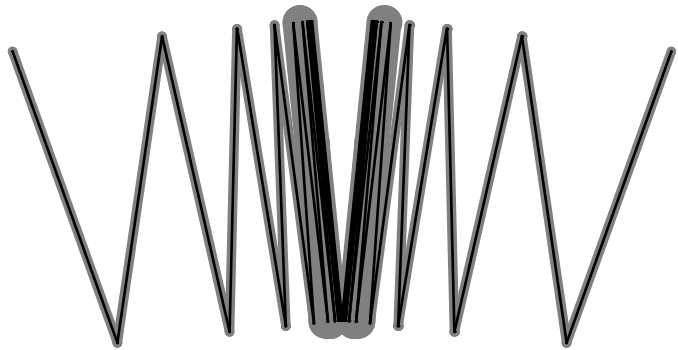

\[
\xycompileto{eqn110}{\xygraph{
!~-{@{-}@[|(2.5)]}
!~:{@{-}@[gray]}
!{<0cm,0cm>;<3.98cm,-0.4cm>:<0.4cm,3.98cm>::}  
!{(0,1)}*{}="01"
!{(0,0)}*{}="00"
!{(0.040,1)}*{}="x01"
!{(0.040,0.6)}*{}="x05"
!{(0.053,0.002)}*{}="x06"
!{(0.040,0)}*{}="x00"
!{(1,1)}*{}="a1"
!{(0.75,0)}*{}="b1"
!{(0.5,1)}*{}="a2"
!{(0.375,0)}*{}="b2"
!{(0.25,1)}*{}="a3"
!{(0.1875,0)}*{}="b3"
!{(0.125,1)}*{}="a4"
!{(0.09375,0)}*{}="b4"
!{(0.0625,1)}*{}="a5"
!{(0.046875,0)}*{}="b5"
!{(0.03125,1)}*{}="a6"
!{(0.0234375,0)}*{}="b6"
!{(0.015625,1)}*{}="a7"
!{(0.01171875,0)}*{}="b7"
!{(0.0078125,1)}*{}="a8"
!{(0.005859375,0)}*{}="b8"
!{(0.00390625,1)}*{}="a9"
!{(0.0029296875,0)}*{}="b9"
!{(0.001953125,1)}*{}="a10"
!{(0.0014648438,0)}*{}="b10"
!{<0cm,0cm>;<-3.98cm,-0.4cm>:<-0.4cm,3.98cm>::}   
!{(0,1)}*{}="z01"
!{(0,0)}*{}="z00"
!{(0.040,1)}*{}="y01"
!{(0.040,0.6)}*{}="y05"
!{(0.053,0.002)}*{}="y06"
!{(0.040,0)}*{}="y00"
!{(1,1)}*{}="za1"
!{(0.75,0)}*{}="zb1"
!{(0.5,1)}*{}="za2"
!{(0.375,0)}*{}="zb2"
!{(0.25,1)}*{}="za3"
!{(0.1875,0)}*{}="zb3"
!{(0.125,1)}*{}="za4"
!{(0.09375,0)}*{}="zb4"
!{(0.0625,1)}*{}="za5"
!{(0.046875,0)}*{}="zb5"
!{(0.03125,1)}*{}="za6"
!{(0.0234375,0)}*{}="zb6"
!{(0.015625,1)}*{}="za7"
!{(0.01171875,0)}*{}="zb7"
!{(0.0078125,1)}*{}="za8"
!{(0.005859375,0)}*{}="zb8"
!{(0.00390625,1)}*{}="za9"
!{(0.0029296875,0)}*{}="zb9"
!{(0.001953125,1)}*{}="za10"
!{(0.0014648438,0)}*{}="zb10"
"x00":@[|(33.0)]"x01"  "y00":@[|(33.0)]"y01"
"x06":@[|(33.0)]"x05"  "y06":@[|(33.0)]"y05"
                       "a1":@[|(9.0)]"b1"  "za1":@[|(9.0)]"zb1"
"00"-"01" "z00"-"z01"  "b1":@[|(9.0)]"a2"  "zb1":@[|(9.0)]"za2"
"a1"-"b1" "za1"-"zb1"  "a2":@[|(9.0)]"b2"  "za2":@[|(9.0)]"zb2"
"b1"-"a2" "zb1"-"za2"  "b2":@[|(9.0)]"a3"  "zb2":@[|(9.0)]"za3"
"a2"-"b2" "za2"-"zb2"  "a3":@[|(9.0)]"b3"  "za3":@[|(9.0)]"zb3"
"b2"-"a3" "zb2"-"za3"  "b3":@[|(9.0)]"a4"  "zb3":@[|(9.0)]"za4"
"a3"-"b3" "za3"-"zb3"  "a4":@[|(9.0)]"b4"  "za4":@[|(9.0)]"zb4"
"b3"-"a4" "zb3"-"za4" 
"a4"-"b4" "za4"-"zb4" 
"b4"-"a5" "zb4"-"za5" 
"a5"-"b5" "za5"-"zb5"
"b5"-"a6" "zb5"-"za6"
"a6"-"b6" "za6"-"zb6"
"b6"-"a7" "zb6"-"za7"
"a7"-"b7" "za7"-"zb7"
"b7"-"a8" "zb7"-"za8"
"a8"-"b8" "za8"-"zb8"
"b8"-"a9" "zb8"-"za9"
"a9"-"b9" "za9"-"zb9"
"b9"-"a10" "zb9"-"za10"
"a10"-"b10" "za10"-"zb10"
}}
\]
\caption{In black the space $X$ as in Figure~\ref{fig:topologist's-sine-curve},
and in gray an example of the neighborhoods $X_{n}$ used
in Example~\ref{exa:JoinNotAANR}.
\label{fig:two-sine-curves-Nasty-why}}
\end{figure}

\begin{example}
\label{exa:Join-is-AANR} 
Joining two copies of the topologist's sine
curve at different point, as indicated in
Figure~\ref{fig:two-sine-curves-Nice}, leads to a space $X$ that
is an AANR, and indeed an AAR. As in Example~\ref{exa:standard-example}
we can approximately retract a rectangle in the surrounding Euclidean
space to $X$. This is illustrated in Figure~\ref{fig:why-join_is-nice}.
\end{example}

\begin{figure}
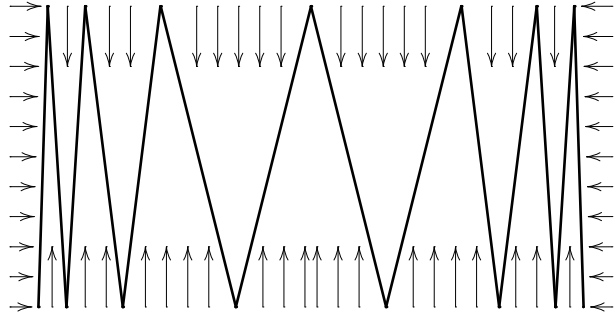

\[
\xycompileto{eqn111}{\xygraph{
!~-{@{-}@[|(2.5)]}
!{<0cm,0cm>;<4.0cm,0.0cm>:<0.0cm,4.0cm>::}  
!{(0,1)}*{}="01"
!{(0,0)}*{}="00"
!{(1,1)}*{}="a1"
!{(0.75,0)}*{}="b1"
!{(0.5,1)}*{}="a2"
!{(0.375,0)}*{}="b2"
!{(0.25,1)}*{}="a3"
!{(0.1875,0)}*{}="b3"
!{(0.125,1)}*{}="a4"
!{(0.09375,0)}*{}="b4"
!{(0.0625,1)}*{}="a5"
!{(0.046875,0)}*{}="b5"
!{(0.03125,1)}*{}="a6"
!{(0.0234375,0)}*{}="b6"
!{(0.015625,1)}*{}="a7"
!{(0.01171875,0)}*{}="b7"
!{(0.0078125,1)}*{}="a8"
!{(0.005859375,0)}*{}="b8"
!{(0.00390625,1)}*{}="a9"
!{(0.0029296875,0)}*{}="b9"
!{(0.001953125,1)}*{}="a10"
!{(0.0014648438,0)}*{}="b10"
!{(0.00,0.0)}*{}="L0"
!{(0.00,0.1)}*{}="L1"
!{(0.00,0.2)}*{}="L2"
!{(0.00,0.3)}*{}="L3"
!{(0.00,0.4)}*{}="L4"
!{(0.00,0.5)}*{}="L5"
!{(0.00,0.6)}*{}="L6"
!{(0.00,0.7)}*{}="L7"
!{(0.00,0.8)}*{}="L8"
!{(0.00,0.9)}*{}="L9"
!{(0.00,1.0)}*{}="LX"
!{(0.07,0.0)}*{}="M0"
!{(0.0727,0.1)}*{}="M1"
!{(0.0754,0.2)}*{}="M2"
!{(0.0781,0.3)}*{}="M3"
!{(0.0808,0.4)}*{}="M4"
!{(0.0835,0.5)}*{}="M5"
!{(0.0862,0.6)}*{}="M6"
!{(0.0889,0.7)}*{}="M7"
!{(0.0916,0.8)}*{}="M8"
!{(0.0943,0.9)}*{}="M9"
!{(0.10,1.0)}*{}="MX"
!{(0.14,0.0)}*{}="U0"
!{(0.14,0.2)}*{}="N0"
!{(0.25,0.0)}*{}="U1"
!{(0.25,0.2)}*{}="N1"
!{(0.32,0.0)}*{}="U2"
!{(0.32,0.2)}*{}="N2"
!{(0.45,0.0)}*{}="U3"
!{(0.45,0.2)}*{}="N3"
!{(0.52,0.0)}*{}="U4"
!{(0.52,0.2)}*{}="N4"
!{(0.59,0.0)}*{}="U5"
!{(0.59,0.2)}*{}="N5"
!{(0.66,0.0)}*{}="U6"
!{(0.66,0.2)}*{}="N6"
!{(0.84,0.0)}*{}="U7"
!{(0.84,0.2)}*{}="N7"
!{(0.91,0.0)}*{}="U8"
!{(0.91,0.2)}*{}="N8"
!{(0.98,0.0)}*{}="U9"
!{(0.98,0.2)}*{}="N9"
!{(0.19,1.0)}*{}="B0"
!{(0.19,0.8)}*{}="K0"
!{(0.33,1.0)}*{}="B1"
!{(0.33,0.8)}*{}="K1"
!{(0.40,1.0)}*{}="B2"
!{(0.40,0.8)}*{}="K2"
!{(0.62,1.0)}*{}="B3"
!{(0.62,0.8)}*{}="K3"
!{(0.69,1.0)}*{}="B4"
!{(0.69,0.8)}*{}="K4"
!{(0.76,1.0)}*{}="B5"
!{(0.76,0.8)}*{}="K5"
!{(0.83,1.0)}*{}="B6"
!{(0.83,0.8)}*{}="K6"
!{(0.90,1.0)}*{}="B7"
!{(0.90,0.8)}*{}="K7"
!{<8.00cm,0cm>+<0.0cm,0cm>;<8.00cm,0cm>+<-4.00cm,-0cm>:<8.00cm,0cm>+<-0cm,4.00cm>::}   
!{(0,1)}*{}="z01"
!{(0,0)}*{}="z00"
!{(1,1)}*{}="za1"
!{(0.75,0)}*{}="zb1"
!{(0.5,1)}*{}="za2"
!{(0.375,0)}*{}="zb2"
!{(0.25,1)}*{}="za3"
!{(0.1875,0)}*{}="zb3"
!{(0.125,1)}*{}="za4"
!{(0.09375,0)}*{}="zb4"
!{(0.0625,1)}*{}="za5"
!{(0.046875,0)}*{}="zb5"
!{(0.03125,1)}*{}="za6"
!{(0.0234375,0)}*{}="zb6"
!{(0.015625,1)}*{}="za7"
!{(0.01171875,0)}*{}="zb7"
!{(0.0078125,1)}*{}="za8"
!{(0.005859375,0)}*{}="zb8"
!{(0.00390625,1)}*{}="za9"
!{(0.0029296875,0)}*{}="zb9"
!{(0.001953125,1)}*{}="za10"
!{(0.0014648438,0)}*{}="zb10"
!{(0.00,0.0)}*{}="zL0"
!{(0.00,0.1)}*{}="zL1"
!{(0.00,0.2)}*{}="zL2"
!{(0.00,0.3)}*{}="zL3"
!{(0.00,0.4)}*{}="zL4"
!{(0.00,0.5)}*{}="zL5"
!{(0.00,0.6)}*{}="zL6"
!{(0.00,0.7)}*{}="zL7"
!{(0.00,0.8)}*{}="zL8"
!{(0.00,0.9)}*{}="zL9"
!{(0.00,1.0)}*{}="zLX"
!{(0.07,0.0)}*{}="zM0"
!{(0.0727,0.1)}*{}="zM1"
!{(0.0754,0.2)}*{}="zM2"
!{(0.0781,0.3)}*{}="zM3"
!{(0.0808,0.4)}*{}="zM4"
!{(0.0835,0.5)}*{}="zM5"
!{(0.0862,0.6)}*{}="zM6"
!{(0.0889,0.7)}*{}="zM7"
!{(0.0916,0.8)}*{}="zM8"
!{(0.0943,0.9)}*{}="zM9"
!{(0.10,1.0)}*{}="zMX"
!{(0.14,0.0)}*{}="zU0"
!{(0.14,0.2)}*{}="zN0"
!{(0.25,0.0)}*{}="zU1"
!{(0.25,0.2)}*{}="zN1"
!{(0.32,0.0)}*{}="zU2"
!{(0.32,0.2)}*{}="zN2"
!{(0.45,0.0)}*{}="zU3"
!{(0.45,0.2)}*{}="zN3"
!{(0.52,0.0)}*{}="zU4"
!{(0.52,0.2)}*{}="zN4"
!{(0.59,0.0)}*{}="zU5"
!{(0.59,0.2)}*{}="zN5"
!{(0.66,0.0)}*{}="zU6"
!{(0.66,0.2)}*{}="zN6"
!{(0.84,0.0)}*{}="zU7"
!{(0.84,0.2)}*{}="zN7"
!{(0.91,0.0)}*{}="zU8"
!{(0.91,0.2)}*{}="zN8"
!{(0.98,0.0)}*{}="zU9"
!{(0.98,0.2)}*{}="zN9"
!{(0.19,1.0)}*{}="zB0"
!{(0.19,0.8)}*{}="zK0"
!{(0.33,1.0)}*{}="zB1"
!{(0.33,0.8)}*{}="zK1"
!{(0.40,1.0)}*{}="zB2"
!{(0.40,0.8)}*{}="zK2"
!{(0.62,1.0)}*{}="zB3"
!{(0.62,0.8)}*{}="zK3"
!{(0.69,1.0)}*{}="zB4"
!{(0.69,0.8)}*{}="zK4"
!{(0.76,1.0)}*{}="zB5"
!{(0.76,0.8)}*{}="zK5"
!{(0.83,1.0)}*{}="zB6"
!{(0.83,0.8)}*{}="zK6"
!{(0.90,1.0)}*{}="zB7"
!{(0.90,0.8)}*{}="zK7"
"B0":"K0" "zB0":"zK0" 
"B1":"K1" "zB1":"zK1" 
"B2":"K2" "zB2":"zK2" 
"B3":"K3" "zB3":"zK3" 
"B4":"K4" "zB4":"zK4" 
"B5":"K5" "zB5":"zK5" 
"B6":"K6" "zB6":"zK6" 
"B7":"K7" "zB7":"zK7" 
"U0":"N0" "zU0":"zN0" 
"U1":"N1" "zU1":"zN1" 
"U2":"N2" "zU2":"zN2" 
"U3":"N3" "zU3":"zN3"
"U4":"N4" "zU4":"zN4"
"U5":"N5" "zU5":"zN5"
"U6":"N6" "zU6":"zN6" 
"U7":"N7" "zU7":"zN7" 
"U8":"N8" "zU8":"zN8" 
"U9":"N9" "zU9":"zN9" 
"L0":"M0" "zL0":"zM0" 
"L1":"M1" "zL1":"zM1" 
"L2":"M2" "zL2":"zM2" 
"L3":"M3" "zL3":"zM3" 
"L4":"M4" "zL4":"zM4" 
"L5":"M5" "zL5":"zM5" 
"L6":"M6" "zL6":"zM6" 
"L7":"M7" "zL7":"zM7" 
"L8":"M8" "zL8":"zM8" 
"L9":"M9" "zL9":"zM9" 
"LX":"MX" "zLX":"zMX" 
"a1"-"b1" "za1"-"zb1" 
"b1"-"a2" "zb1"-"za2" 
"a2"-"b2" "za2"-"zb2" 
"b2"-"a3" "zb2"-"za3" 
"a3"-"b3" "za3"-"zb3" 
"b3"-"a4" "zb3"-"za4" 
"a4"-"b4" "za4"-"zb4" 
}}
\]
\caption{In illustration of the approximate retraction of a rectangle onto
the space in Figure~\ref{fig:two-sine-curves-Nice}. 
\label{fig:why-join_is-nice}}
\end{figure}

\section{PAANR spaces}
\label{sec:PAANR-spaces}

We wish to rework Section~\ref{sec:AANR-spaces} for pointed compacta.
In contrast to the situation regarding ANR spaces, the PAANR property
will depend on the choice of point. We already encountered this dependence
when studying, in \cite{LoringWeaklyProjective}, pointed approximative
absolute retracts (PANR).

\begin{defn}
\label{def:PAANR}A pointed compactum $\left(X,x_{0}\right)$ is a
\emph{pointed approximative absolute neighborhood retract (PAANR)}
if, for every homeomorphic embedding $\theta:X\rightarrow Y$ of $X$
into a compact metric space $(Y,d)$, and for every $\epsilon>0$,
there exists $\delta>0$ and a continuous function $r:U_{\delta}\rightarrow X$
so that 
\[
r(\theta(x_{0}))=x_{0}
\]
and
\[
d(r\circ\theta(x),x)\leq\epsilon
\]
for all $x$ in $X,$ where $U_{\delta}$ is as in (\ref{eq:def-of-U_delta}). 
\end{defn}

We could just as well have asked that $\left(Y,y_{0}\right)$ be a
pointed compactum with compatible metric $d$ and that $\theta$ and
$r$ be pointed maps. As before we wish to replace the $U_{\delta}$
with more general closed sets that decrease to $\theta(X)$. The sets
need not be neighborhoods of $\theta(X)$, although later we will
require this when we devise a method for proving that a closed subset,
with chosen point, of an AR is a PAANR.

\begin{thm}
\label{thm:PAANRviaApproxSplit} 
A pointed compactum $\left(X,x_{0}\right)$ is a PAANR if, and only
if, for every continuous embedding $\theta:X\rightarrow Y$ of $X$ into
a compactum $Y$, and for every sequence $Y_{1}\supseteq Y_{2}\supseteq\cdots$
of closed subsets of $Y$ with $\bigcap Y_{n}=\theta(X),$ there exists a
sequence of continuous functions $r_{n}:Y_{n}\rightarrow X$ so that
\[
r_{n}(\theta(x_{0}))=x_{0}
\]
and
\[
\lim_{n\rightarrow\infty}r_{n}(\theta(x))=x
\]
uniformly over $x$ in $X.$ 
\end{thm}

\begin{proof}
The reverse implication is once again trivial.

Assume $X$ is an PAANR and that the $Y_{n}$ are given. Let $d$
be a compatible metric on $Y.$ For a given $k$ we know there is
a $\delta>0$ and continuous map $r:U_{\delta}\rightarrow X$
so that $r(\theta(x_{0}))=x_{0}$ and 
\[
d(r(\theta(x)),x)\leq\frac{1}{k}
\]
for all $x$ in $X.$ The same argument used for
Proposition~\ref{pro:AANRviaApproxSplit} shows we have
$Y_{n_{k}}\subseteq U_{\frac{1}{k}}$ for some $n$
and we can again use the restrictions of $r$ to various $Y_{n_{k}}$.
\end{proof}

\begin{thm}
\label{thm:PAANRviaApproxExten} 
A pointed compactum $\left(X,x_{0}\right)$  is a PAANR if, and
only if, for every closed subset $Y$ of a compact metrizable space
$Z$, for every $y_{0}\in Y$, for every sequence
$Y_{1}\supseteq Y_{2}\supseteq\cdots$ of closed subsets of $Z$ with
$\bigcap Y_{n}=Y,$ and for every continuous function
$\lambda:Y\rightarrow X$ for which $\lambda(y_{0})=x_{0}$,
there is a sequence of continuous
functions $\lambda_{n}:Y_{n}\rightarrow X$ so that
\[
\lambda_{n}(y_{0})=\lambda(y_{0})
\]
and 
\[
\lim_{n\rightarrow\infty}\lambda_{n}(y)=\lambda(y)
\]
uniformly for $y$ in $Y.$ To summarize in a diagram: 
\begin{equation}
\xycompileto{eqn118}{\xymatrix{
 &	(Z, y_0)\\
 & (Y_n, y_0)\ar@{-->}[dl] _(0.4){ \lambda_n }  \ar@{^{(}->}[u] \\
 (X, x_0) & (Y,y_0) \ar[l] ^(0.4){\lambda} \ar@{^{(}->}[u]  
}
}
\label{eqn:PAANRextensionDiagram}
\end{equation}
\end{thm}

\begin{proof}
We need only modify in a few places the proof of
Proposition~\ref{pro:AANRviaApproxExten}.

In the proof of the reverse implication, the additional assumption
$\lambda(y_{0})=x_{0}$ means that $y_{0}$ and $x_{0}$ get identified
in the push-out, or more precisely $\iota_{X}(x_{0})=\iota_{Z}(y_{0}).$
Instead of invoking Proposition~\ref{pro:AANRviaApproxSplit} we
invoke Theorem~\ref{thm:PAANRviaApproxSplit}, which gives us 
\[
\rho_{n}:\iota_{X}(X)\cup\iota_{Z}\left(Y_{n}\right)\rightarrow X
\]
so that
\[
\lim_{n\rightarrow\infty}\rho_{n}\circ\iota_{X}(x)=x
\]
uniformly over $x$ in $X$ and 
\[
\rho_{n}\circ\iota_{X}(x_{0})=x_{0}.
\]
As before, $r_{n}:Y_{n}\rightarrow X$ is defined by
\[
r_{n}(y)=\rho_{n}(\iota_{Z}(y))
\]
 and we get the same uniform convergence $r_{n}(y)\rightarrow\lambda(y)$,
but additionally we find
\[
r_{n}(y_{0})
=\rho_{n}(\iota_{Z}(y_{0}))
=\rho_{n}(\iota_{X}(x_{0}))
=x_{0}.
\]

Going in the other direction, we started with $\theta:X\rightarrow Y$
an embedding, and now find continuous functions $r_{n}:Y_{n}\rightarrow X$
with 
\[
\lim_{n\rightarrow\infty}r_{n}(y)=\theta^{-1}(y)
\]
and
\[
r_{n}(\theta(x_{0}))=\theta^{-1}(\theta(x_{0}))
\]
and so get the needed additional conclusion $r_{n}(\theta(x_{0}))=x_{0}.$
\end{proof}

\begin{thm}
\label{thm:ARtestForPAANR} 
Suppose $X$ is a closed subset of $Q$ where $Q$ is an absolute retract,
and that $x_{0}$ is a point in $X$. Suppose
$X_{1}\supseteq X_{2}\supseteq\cdots$ are closed subset with
$\bigcap X_{n}=X$ and where for each $n$ the interior of $X_{n}$
contains $X$. Then $\left(X,x_{0}\right)$ is an AANR if and only if
there is a sequence of continuous functions $r_{n}:X_{n}\rightarrow X$
so that
\[
r_{n}(x_{0})=x_{0}
\]
and
\[
\lim_{n\rightarrow\infty}r_{n}(x)=x
\]
uniformly over $x$ in $X$. 
\end{thm}

\begin{proof}
The proof of Proposition~\ref{pro:ARtestForAANR} can be modified
as follows, where it is again only the backwards implication that
involves any work. We are starting with the additional assumption
that $r_{n}(x_{0})=x_{0}$ and so at the end of the proof we can calculate
\[
\rho_{\ell}\left(\theta(x_{0})\right)
=\rho_{k_{n}}\left(\theta(x_{0})\right)
=r_{n}(\alpha(\theta(x_{0})))
=r_{n}(x_{0})
=x_{0}.
\]

\end{proof}
For the record, we have an obvious implication. 

\begin{prop}
If $\left(X,x_{0}\right)$ is a PAANR then $X$ is an AANR. 
\end{prop}

The reverse implication fails. The example is the same example that
showed in \cite{LoringWeaklyProjective} that a pointed compacta can
fail to be a pointed approximative absolute retract (PAAR) while the
underlying space is AAR.

\begin{figure}
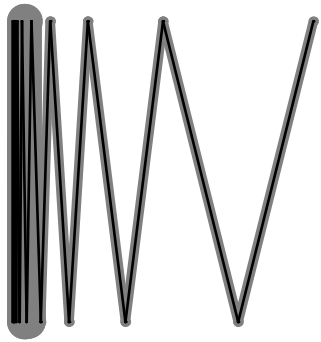

\[
\xycompileto{eqn113}{\xygraph{
!~-{@{-}@[|(2.5)]}
!~:{@{-}@[gray]}
!{<0cm,0cm>;<4.0cm,0.0cm>:<0.0cm,4.0cm>::}  
!{(0,1)}*{}="01"
!{(0,0)}*{}="00"
!{(0.040,1)}*{}="x01"
!{(0.040,0.6)}*{}="x05"
!{(0.053,0.002)}*{}="x06"
!{(0.040,0)}*{}="x00"
!{(1,1)}*{}="a1"
!{(0.75,0)}*{}="b1"
!{(0.5,1)}*{}="a2"
!{(0.375,0)}*{}="b2"
!{(0.25,1)}*{}="a3"
!{(0.1875,0)}*{}="b3"
!{(0.125,1)}*{}="a4"
!{(0.09375,0)}*{}="b4"
!{(0.0625,1)}*{}="a5"
!{(0.046875,0)}*{}="b5"
!{(0.03125,1)}*{}="a6"
!{(0.0234375,0)}*{}="b6"
!{(0.015625,1)}*{}="a7"
!{(0.01171875,0)}*{}="b7"
!{(0.0078125,1)}*{}="a8"
!{(0.005859375,0)}*{}="b8"
!{(0.00390625,1)}*{}="a9"
!{(0.0029296875,0)}*{}="b9"
!{(0.001953125,1)}*{}="a10"
!{(0.0014648438,0)}*{}="b10"
"x00":@[|(33.0)]"x01"  
"x06":@[|(33.0)]"x05"  
"a1":@[|(9.0)]"b1"  
"00"-"01"   "b1":@[|(9.0)]"a2"  
"a1"-"b1"   "a2":@[|(9.0)]"b2"  
"b1"-"a2"   "b2":@[|(9.0)]"a3"  
"a2"-"b2"   "a3":@[|(9.0)]"b3"  
"b2"-"a3"   "b3":@[|(9.0)]"a4"  
"a3"-"b3"   "a4":@[|(9.0)]"b4"  
"b3"-"a4"  
"a4"-"b4"  
"b4"-"a5"  
"a5"-"b5" 
"b5"-"a6" 
"a6"-"b6" 
"b6"-"a7" 
"a7"-"b7" 
"b7"-"a8" 
"a8"-"b8" 
"b8"-"a9"
"a9"-"b9"
"b9"-"a10"
"a10"-"b10"
}}
\]
\caption{In black the space $X$ as in Figure~\ref{fig:topologist's-sine-curve},
and in gray an example of the neighborhoods $X_{n}$ used in
Example~\ref{exa:basePointMatters-1}
\label{fig:pointed-sine-curve-Nasty}}
\end{figure}

\begin{example}
\label{exa:basePointMatters}
Consider the topologist's sine curve $X$ as illustrated in
Figure~\ref{fig:topologist's-sine-curve}(a), and the point $x_{1}$ as
in Figure~\ref{fig:topologist's-sine-curve}(b).  The approximate
retractions shown in Figure~\ref{fig:topologist's-sine-curve-1} all
fix $x_{1}$ and so $\left(X,x_{1}\right)$ is a PAANR. 
\end{example}

\begin{example}
\label{exa:basePointMatters-1}
Consider the topologist's sine curve $X$ with the point $x_{0}$ from
Figure~\ref{fig:topologist's-sine-curve}(b).  Consider the neighborhoods
of $X$ indicated in Figure~\ref{fig:pointed-sine-curve-Nasty}.  These
are path connected, and as our approximate retracts are required to
fix $x_{0}$ all of the neighborhood must be mapped into the left edge
of $X$. This is incompatible with the requirement that we approximately
fix $x_{1}$ and so $\left(X,x_{0}\right)$ is not a PAANR. 
\end{example}

An alternative to ``ignoring the special point'' is to appoint an
uninteresting point to fill the ``special role.'' Starting with compact
space $X$ we can go ahead and take the one-point compactification
$\alpha X=X\cup\{\infty\}$, which leads to a compact space which
is the old space plus a new isolated point. This leads us to the following,
which will look a lot more interesting when dualized to be about
$C^{*}$-algebras,
in Theorem~\ref{thm:WSP1_iff_tilde_A_isWSP}

\begin{prop}
Suppose $\left(X,x_{0}\right)$ is a pointed compactum. Then $X$
is an AANR if, and only if, the extension property in
Theorem~\ref{thm:PAANRviaApproxExten} holds in the special case
where $y_{0}$ is an isolated point in $Y$.
\end{prop}

\begin{proof}
Assume first $X$ is an AANR. We are given $\lambda:Y\rightarrow X$
where $Y$ is compact and $y_{0}$ is isolated, and we are given
$\lambda:Y\rightarrow X$ with $\lambda(y_{0})=x_{0}$, and finally
have $Y_{1}\supseteq Y_{2}\supseteq\cdots$ compact sets with
$Y=\bigcap_{n}Y_{n}$. Since $y_{0}$ is isolated in $Y$ and $Y$
is a compact subset in the compactum $Z$, there
are disjoint sets $U$ and $V$ open in $Z$ with 
\[
\{y_{0}\}=Y\cap U
\]
and
\[
Y\setminus\{y_{0}\}=Y\cap V.
\]
The compact sets $Y_{n}\setminus\left(U\cap V\right)$ are decreasing
with intersection
\[
\bigcap_{n}Y_{n}\setminus\left(U\cap V\right)
=Y\setminus\left(U\cap V\right)
=\emptyset
\]
so for some $N$, when $n\geq N$ we have $Y_{n}=A_{n}\cup B_{n}$
where $A_{n}=Y_{n}\cap U$ and $B_{n}=Y_{n}\cap V$. Without loss
of generality, $N=1$, so we have $Y_{n}$ written as the disjoint
union of closed subsets $A_{n}$ and $B_{n}$ with both forming decreasing
chains and
\[
\bigcap A_{n}=Y_{n}\setminus\{y_{0}\}
\]
and 
\[
\bigcap B_{n}=\{y_{0}\}.
\]
We can define $\lambda_{n}:Y_{n}\rightarrow X$ by 
$\lambda_{n}(y)=\lambda(y_{0})$ for all $y$ in $B_{n}$ and use the
fact that $X$ is AAR to define $\lambda_{n}$ on $A$ so that
$\lambda_{n}(y)\rightarrow\lambda(y)$ uniformly over $x$ in
$B_{n}$. This is the desired approximate extension that is an
exact extension on $y_{0}$.

Now assume the specialized version of the approximate extension property
holds and that we are given $\lambda:Y\rightarrow X$ for $Y$ a closed
subset of a compactum $Z$. We also have decreasing closed $Y_{n}$
with intersection $Y$. We add to $Z$ an isolated point $\infty$
and consider the closed subsets $Y\cup\{\infty\}$ and
$Y_{n}\cup\{\infty\}$ of $Z\cup\{\infty\}$. We can extend $\lambda$
to a map $\bar{\lambda}$ from $Y\cup\{\infty\}$ to $X$ by
arbitrarily selecting $x_{0}$ in $X$ and setting $\lambda(\infty)=x_{0}$.
Then there are continuous functions
$\bar{\lambda}_{n}:Y_{n}\cup\{\infty\}\rightarrow X$ with
$\bar{\lambda}_{n}(\infty)=x_{0}$ 
and
$\bar{\lambda}_{n}(y)\rightarrow\bar{\lambda}(y)$
uniformly over $Y\cup\{\infty\}$. The desired functions are
the restrictions of the $\bar{\lambda}_{n}$ to the sets $Y_{n}$.
\end{proof}

\section{Two flavors of weak semiprojectivity}

The analog of being an ANR compactum for a $C^{*}$-algebra is that
it is unital and semiprojective. Indeed, Blackadar's definition of
semiprojectivity \cite{Blackadar-shape-theory} is modeled on the
non-approximative version of our Proposition~\ref{pro:AANRviaApproxExten}.
A $C^{*}$-algebra $A$ will semiprojective if we can solve the partial
lifting problem indicated here:
\begin{equation}
\xycompileto{eqn114}{\xymatrix{
 &	B \ar@{->>}[d] ^{\rho_n} \\
 & C_n \ar@{->>}[d] ^{\pi_n} \\
 A \ar[r] ^{\varphi} \ar@{-->}[ru] ^{\psi_n} & C   
}
}
\label{eqn:WSPbyLiftingDiagram}
\end{equation}

\begin{defn}
\label{def:WSP}
A separable $C^{*}$-algebra $A$ is \emph{semiprojective} (SP) if given
a $*$-homomorphism $\varphi:A\rightarrow B/J$, with $B$ a separable
$C^{*}$-algebra with ideal $J=\overline{\bigcup_{n}J_{n}}$ and
$J_{1}\triangleleft J_{2}\triangleleft\cdots$ increasing ideals in
$B$, there exist for some $n$ a $*$-homomorphism
$\psi_{n}:A\rightarrow B/J_{n}$ so that 
$\pi_{n}\circ\psi_{n}(a)=a$ for all $a$ in $A$. 
\end{defn}

We are using the $\pi_{n}$ to be the surjection defined by $\pi_{n}(b+J_{n})=b+J$.

Weak semiprojectivity can be found by weakening this partial lifting
problem in two seemingly different ways. We can either restrict the
allowed $B$ and $J_{n}$ and keep the exact lifting requirement $\pi_{n}\circ\psi_{n}(a)=a$,
or we can leave the allowed $B$ and $J_{n}$ alone and only ask that
$\pi_{n}\circ\psi_{n}(a)\rightarrow a$ for all $a$ in $A$. 

\begin{rem}
If $A$ is commutative we can define\emph{ weak semiprojectivity within
the commutative category}. We can do the same for all the variations
on semiprojectivity that follow. 
\end{rem}

\begin{defn}
A separable $C^{*}$-algebra $A$ is \emph{weakly semiprojective}
(WSP) if given a $*$-homomorphism $\varphi:A\rightarrow B/J$, with
$B$ a separable $C^{*}$-algebra with ideal $J=\overline{\bigcup_{n}J_{n}}$
and $J_{1}\triangleleft J_{2}\triangleleft\cdots$ increasing ideals
in $B$, there exist a sequence of $*$-homomorphism $\psi_{n}:A\rightarrow B/J_{n}$
so that $\pi_{n}\circ\psi_{n}(a)\rightarrow a$ for all $a$ in $A$. 
\end{defn}

It was shown in \cite{EilersLoringContingenciesStableRelations} that
this is equivalent to the original definition
\cite{Loring-lifting-perturbing} of weak semiprojectivity. In that
formulation, $B$ is always an infinite product $B=\prod B_{n}$ and
$J_{n}=B_{1}\oplus\cdots\oplus B_{n}$.  It then is possible to
interleave any sequence of approximate partial liftings into an
exact lifting to $B$. 

A $C^{*}$-algebra is semiprojective if and only if its unitization
is $\widetilde{A}$ is semiprojective, as was shown in
\cite{Blackadar-shape-theory}.  For weak semiprojectivity this fails.
We show this in Example~\ref{exa:de-unitizationCanFail},
with the aid of the following lemmas.

\begin{prop}
\label{pro:A-WSP=00003D=00003D>A-tilde-WSP}
If $A$ is separable
and WSP then $\widetilde{A}$ is WSP.
\end{prop}

\begin{proof}
Assume $A$ is WSP and that we have $\varphi:\widetilde{A}\rightarrow B/J$
and the chain of ideals $J_{n}$. For some $n$ we can lift $\varphi(1)$
to a projection $p$ in $B/J_{n}$. Here we have used the semiprojectivity
of $\mathbb{C}$ (Lemma 4.2.2 in \cite{Loring-lifting-perturbing})
or the usual argument involving functional calculus and lifting the
relations $x^{*}=x^{2}=x$. Consider $C=p\left(B/J_{n}\right)p$,
which is a unital $C^{*}$-subalgebra of $B/J_{n}$, and
$K_{m}=J_{m}\cap C$ is an ideal of $C$, and the diagram
\[
\xycompileto{eqn116}{
\xymatrix{
	&	& C/K_m \ar[d] ^{\hat{\pi}_m} \ar[r] ^{\alpha}
			& B/J_m \ar[d] ^{\pi_m}
\\
A \ar[r] \ar[r] _{\iota}
	& {\widetilde{A}} \ar[r]^(0.45){\varphi_0} \ar@/_0.4cm/[rr] _{\varphi}
		& C/K \ar[r] \ar[r] ^{\beta}
			& B/J 
}
}
\]
where
\[
K=\bigcap_{m\geq n}K_{m}=J\cap C
\]
and the horizontal maps are induced by the inclusion of $C$ into
$B$. Applying the weak semiprojectivity of $A$ we find
$\psi_{m}:A\rightarrow C/K_{m}$ with
$\hat{\pi}_{m}\circ\psi_{m}\rightarrow\varphi_{0}\circ\iota$.
Since $C$ is unital we can extend this to
$\widetilde{\psi}_{m}:\widetilde{A}\rightarrow C/K_{m}$
with $\hat{\pi}_{m}\circ\widetilde{\psi}_{m}\rightarrow\varphi$.
Finally, we use $\alpha\circ\widetilde{\psi}_{m}$ as the needed
approximate lifts $A\rightarrow B/J_{m}$ for $m\geq n,$ filling
in with the zero map for $m<n$.
\end{proof}

\begin{prop}
\label{pro:WP_implies_WSP}
Let $A$ be a separable $C^{*}$-algebra. If $A$ is weakly projective
then $A$ is weakly semiprojective.
\end{prop}

\begin{proof}
The definition given in \cite{LoringWeaklyProjective} of $A$ being
weakly projective (WP) is that we can approximately solve a lifting
problem
\[
\xycompileto{eqn117}{
\xymatrix{
	& B \ar[d] ^{\pi}
\\
A \ar[r] \ar@{-->}[ru] & B/J
}
}
\]
and so we can easily solve the WSP lifting
problem.
\end{proof}

We trust that the reader has noticed that the
diagram~(\ref{eqn:WSPbyLiftingDiagram})
is the dual of the diagram~(\ref{eqn:PAANRextensionDiagram}). There
is much that can said about the connection between AANR spaces and
WSP $C^{*}$-algebras---see \cite[2.8-9]{Blackadar-shape-theory}
and \cite[Theorem 1.3]{sorensen2011characterization}---but all that
really concerns us at the moment is that if we want $C_{0}(X)$ to
be WSP then a necessary condition is that $X$ be a PAANR.

\begin{thm}
\label{thm:WSP=00003D=00003D>PAANR}
Suppose $X$ is a locally compact, metrizable space. 
\begin{enumerate}
\item
If $C_{0}(X)$ is WSP then $\left(\alpha X,\infty\right)$ is a PAANR. 
\item
If $\left(\alpha X,\infty\right)$ is a PAANR then $C_{0}(X)$ is
WSP within the commutative category.
\end{enumerate}
\end{thm}

\begin{proof}
Theorem~\ref{thm:PAANRviaApproxExten} tells us that to show
$\left(\alpha X,\infty\right)$ is a PAANR, we need to handle the
approximate extension as in diagram~(\ref{eqn:PAANRextensionDiagram}).
In terms of the induced $*$-homomorphisms, $\lambda$ and the
inclusions give us the diagram
\begin{equation}
\xycompileto{eqn119}{\xymatrix{
 &	C_0(Z, y_0) \ar@{->>}[d]\\
 & C_0(Y_n, y_0)  \ar@{->>}[d] ^{\rho_n}  \\
(\alpha X, \infty ) \ar[r] _(0.48){\lambda^*} \ar@{-->}[ru] ^{\varphi_n}
	 & C_0(Y,y_0) 
}
}\label{eqn:commutativeLift}
\end{equation}
where Definition~\ref{def:WSP} provides us with $\varphi_{n}$ as
in the diagram with
\[
\lim_{n\rightarrow\infty}\left\Vert \rho_{n}\circ\varphi_{n}(f)-f\circ\lambda\right\Vert
=0
\]
for all $f$ in $C_{0}(\alpha X,\infty)$. Since $\varphi_{n}$ is
induced by some map $\lambda_{n}$ of pointed compacta, this is saying
\[
\lim_{n\rightarrow\infty}\sup_{y\in Y}\left| f(\lambda_{n}(y))-f(\lambda(y))\right|
=0.
\]
As this is true for all $f$, we conclude $\lambda_{n}(y)\rightarrow\lambda(y)$
uniformly over $y$ in $Y$. For the second claim, we note simply
that in the commutative situation, up to isomorphism the only liftings
we need are those in (\ref{eqn:commutativeLift}).
\end{proof}

\begin{example}
\label{exa:de-unitizationCanFail}
Consider $A_{0}=C_{0}\left(X\setminus\{x_{0}\}\right)$ and
$A_{1}=C_{0}\left(X\setminus\{x_{0}\}\right)$, where $X$ is the
topologist's sine curve and the named points are as in
Figure~\ref{fig:topologist's-sine-curve}(b).  The claim is that
$A_{1}$ and $\widetilde{A_{0}}\cong\widetilde{A_{1}}\cong C(X)$
are all WSP, while $A_{0}$ is not WSP. In \cite{LoringWeaklyProjective}
we showed that $A_{1}$ is WP, and so
Propositions~\ref{pro:A-WSP=00003D=00003D>A-tilde-WSP} and~\ref{pro:WP_implies_WSP}
imply that $\widetilde{A_{1}}$ is WSP.  Of course this means
$\widetilde{A_{0}}$ is WSP. Example~\ref{exa:basePointMatters-1}
shows $\left(X,x_{0}\right)$ is not a PAANR, so by
Lemma~\ref{thm:WSP=00003D=00003D>PAANR}, $A_{0}$ is not WSP.
\end{example}

We do find that the WSP property behaves well with direct sums.

\begin{thm}
\label{thm:sumWSP_iff_eachWSP}
Suppose $A_{1}$ and $A_{2}$ are
separable $C^{*}$-algebras. Then $A_{1}\oplus A_{2}$ is WSP if,
and only if, both $A_{1}$ and $A_{2}$ are WSP.
\end{thm}

\begin{proof}
Suppose $A_{1}\oplus A_{2}$ is weakly semiprojective and that we
are given $\varphi:A_{1}\rightarrow B/J$ where $J$ is an ideal,
etc. We utilize the inclusion $\iota_{j}:A_{j}\rightarrow A_{1}\oplus A_{2}$
and the projection $\gamma_{j}:A_{1}\oplus A_{2}\rightarrow A_{j}$
in considering the diagram
\[
\xycompileto{eqn120}{\xymatrix{
	&	&	& B \ar[d]
\\	
	&	&	& B/J_n \ar[d] ^{\pi_n} 
\\
A_j \ar[r] ^(0.4){\iota_j} 
	& A_1 \oplus A_2 \ar[r] ^(0.6){\gamma_j}
		& A_j \ar[r] ^{\varphi}
			& B/J
}
}
\]
We have $\varphi_{n}:A_{1}\oplus A_{2}\rightarrow B/J_{n}$
with $\pi_{n}\circ\varphi_{n}(x)\rightarrow\varphi\circ\gamma_{j}(x)$
for all $x$ in $A_{1}\oplus A_{2}$. Therefore
\[
\lim\left\Vert \varphi(a)-\pi_{n}\circ\varphi_{n}\circ\iota_{j}(a)\right\Vert 
=
\lim\left\Vert \varphi\circ\gamma_{j}\left(\iota_{j}(a)\right)-\pi_{n}\circ\varphi_{n}\left(\iota_{j}(a)\right)\right\Vert 
=
0.
\]

Now assume $A_{1}$ and $A_{2}$ are weakly semiprojective and that
we are given $\varphi:A_{1}\oplus A_{2}\rightarrow B/J$ and so forth.
Let $h_{1}$ and $h_{2}$ be strictly positive elements in $A_{1}$
and $A_{2}$ and consider $\varphi((h_{1},0))$ and $\varphi((0,h_{2}))$.
These are orthogonal positive elements, and so can be lifted to orthogonal
positive elements $k_{1}$ and $k_{2}$ in $B$. The argument here
depends on the projectivity of $C_{0}(0,1]\oplus C_{0}(0,1]$, which
is equivalent to the argument that orthogonal, positive contractions
lift to orthogonal, positive contractions, Proposition 10.1.10 in
\cite{Loring-lifting-perturbing}. Inside $B$ we form
$B_{j}=\overline{k_{j}Bk_{j}}$, that is the hereditary subalgebra
generated by $k_{j}$, and as these two $C^{*}$-subalgebras are
orthogonal, we have the copy $B_{1}+B_{2}$ of $B_{1}\oplus B_{2}$ in
$B$. The image of $B_{1}$ under $\pi$ includes
$\varphi((h_{1},0))(B/J)\varphi((h_{1},0))$ and so all of
$\varphi(A_{1}\oplus0)$. Similarly $\pi(B_{2})$ contains
$\varphi(0\oplus A_{2})$ and so the image of $\varphi$ is contained
in the image of $\pi$. If we consider $\pi$ restricted to $B_{j}$
we find it has kernel
\[
J\cap\overline{k_{j}Bk_{j}}=\overline{k_{j}Jk_{j}},
\]
where for the inclusion of left into right we use the approximate
identity $k_{j}^{t}$ for $B_{j}$. We have also a chain of ideals
$\overline{k_{j}J_{n}k_{j}}$ with intersection
$\overline{k_{j}Jk_{j}}$.
We have then a commutative diagram 
\[
\xycompileto{eqn121}{\xymatrix@R=1.2cm@C=1.8cm{
	& B_1/\overline{k_1J_nk_1} \oplus B_2/\overline{k_2J_nk_2} 
		\ar[d] ^{\pi_n \oplus \pi_n} \ar[r] ^(0.7){\alpha_n}
		& B/J_n \ar[d] ^{\pi_n}
\\
A_1 \oplus A_2 \ar[r] ^(0.33){\varphi_0} \ar@/_0.7cm/[rr] _{\varphi} 
	& B_1 \left / \overline{k_1Jk_1} \right . \oplus B_2 \left / \overline{k_2Jk_2} \right . 
		\ar[r] ^(0.7){\alpha}
		& B/J
}
}
\]
and it is evident
how we can use approximate lifting of maps from $A_{1}$ and
$A_{2}$ to create the desired approximate lifting of the maps
from $A_{1}\oplus A_{2}$.
\end{proof}

\begin{rem}
It is important to note that we used the fact that $\overline{hBh}$
was again a $C^{*}$-algebra when $h$ in $B$ is positive, and will
lose this technique when we restrict to lifting problems involving
only unital $C^{*}$-algebras.
\end{rem}

\begin{defn}
\label{def:WSP1}
Let $A$ be a separable $C^{*}$-algebra $A$, not necessarily unital.
We say $A$ is \emph{weakly semiprojective with respect to unital
$C^{*}$-algebras} (WSP1) if we can solve the partial approximate
lifting problem in Definition~\ref{def:WSP} in the special case where
$B$ is a unital separable $C^{*}$-algebra. 
\end{defn}

There is already a definition of weakly semiprojective with respect
to the class of all unital $C^{*}$-algebras, Definition~5.2 in
\cite{EilersLoringContingenciesStableRelations},
but it is equivalent to the one given here.

\begin{rem}
Even if $A$ has a unit, we are not requiring $\varphi$ or the $\varphi_{n}$
to be unital. In particular, if we have $\varphi_{n_{k}}:A\rightarrow B/J_{n_{k}}$
with $\pi_{n_{k}}\circ\psi_{n_{k}}(a)\rightarrow a$ then we can pad
this out with zero maps, and use the intermediate quotient maps
$B/J_{n}\rightarrow B/J_{n+1}$,  to get the required sequence
$\varphi_{n}:A\rightarrow B/J_{n}$.  This was true for weak
semiprojectivity. A common formulation of weak semiprojectivity is that
given $a_{1},\dots,a_{r}$ and $\epsilon>0$ there is 
$\psi:A\rightarrow B/J_{n}$ for some $n$ with
$\left\Vert \pi_{n}\circ\psi(a_{j})-\varphi(a_{j})\right\Vert <\epsilon$.
\end{rem}

\begin{thm}
\label{thm:WSPviaProduct/Sum}
If $A$ is a separable $C^{*}$-algebra then following are equivalent:
\begin{enumerate}
\item
$A$ is WSP1;
\item
the partial approximate lifting problem in Definition~\ref{def:WSP}
can be solved whenever $B$ is a unital $C^{*}$-algebra (so $B$
is not necessarily separable);
\item
Given a $*$-homomorphism
$\varphi:A\rightarrow\left.\prod B_{k}\right/\bigoplus B_{k}$
with $B_{1},B_{2},\dots$ a sequence of unital $C^{*}$-algebras,
there is a $*$-homomorphism $\overline{\varphi}:A\rightarrow\prod B_{k}$
so that $\kappa\circ\overline{\varphi}=\varphi$. Here the sum and
products are indexed by $\mathbb{N}$ and $\kappa$ is the quotient
map.
\end{enumerate}
\end{thm}

\begin{proof}
Clearly (2) implies (1). For the reverse, assume we are facing
$\varphi:A\rightarrow B/J$
and so forth with $B$ not countable. Take a countable dense subset
in $A$, push these forward with $\varphi$ to the quotient and then
take a random lift of this set to a countable set in $B$. Let $\hat{B}$
be the $C^{*}$-algebra generated by this set and the unit, and
$\hat{J}_{n}=J_{n}\cap\hat{B}.$
These nested ideals of $\hat{B}$ have intersection $\hat{J}=J\cap\hat{B}$
and we can factor $\varphi$ through $\hat{B}/\hat{J}$, which we
treat as a subset of $B/J$, leading us to
\[
\xycompileto{eqn115}{\xymatrix{
	&\hat{B}/\hat{J}_n \ar[r] \ar[d]
		&B/J_n \ar[d]
\\
A \ar[r]^{\varphi_0} \ar@/_0.4cm/[rr] _{\varphi}
	& \hat{B}/\hat{J} \ar[r]
		& B/J
}}
\]
which commutes and has $\hat{B}$ separable and unital. The approximate
partial lifts of $\varphi_{0}$ can be composed with the inclusions
of the $\hat{B}/\hat{J_{n}}$ into the $B/J_{n}$, solving the problem.

For the equivalence of (2) and (3) we note that the proof of Theorem
3.1 in \cite{EilersLoringContingenciesStableRelations} works just
as written in the case where the various target $C^{*}$-algebras
are unital. What is essential is that the class of unital
$C^{*}$-algebras is closed under quotients and countable direct sums.
\end{proof}

\begin{thm}
\label{thm:WSP1_iff_tilde_A_isWSP}
Suppose $A$ is a separable $C^{*}$-algebra.
Then $A$ is WSP1 if, and only if, $\widetilde{A}$ is WSP.
\end{thm}

\begin{proof}
The proof of the forward implication of
Proposition~\ref{pro:A-WSP=00003D=00003D>A-tilde-WSP}
works here to give the forward implication. Notice that we started
with $B$ possibly lacking a unit, but then cut down by a projection
to create a $C^{*}$-subalgebra $C$ that was unital.

Now suppose $\widetilde{A}$ is WSP and we are given a map from $A$
over to a unital situation we can extend to $\widetilde{A}$ using
the unit in $B$. Lift the bigger $C^{*}$-algebra and the smaller
goes along from the ride.

Assume $\widetilde{A}$ is WSP and we have $\varphi:A\rightarrow B/J$
with $B$ unital. Since $B/J$ is also unital, we can extend $\varphi$
to a map $\hat{\varphi}:\widetilde{A}\rightarrow B/J$ and so we arrive
at this diagram 
\[
\xycompileto{eqn122}{\xymatrix{
	&	& B \ar[d]
\\
	&	& B/J_n \ar[d] ^{\pi_n}
\\
A \ar[r] ^{\iota} \ar@/_0.4cm/[rr] _{\varphi}
	& {\widetilde{A}} \ar[r] ^{\hat{\varphi}} \ar@{-->}[ru] ^{\psi_n}
		&B/J
}
}
\]
where we use the assumption on
$\widetilde{A}$ to find $\psi_{n}:\widetilde{A}\rightarrow B/J_{n}$
with $\pi_{n}\circ\psi_{n}\rightarrow\hat{\varphi}$. The desired
approximate lifts are the compositions $\iota\circ\psi_{n}$.
\end{proof}

\begin{thm}
Suppose $A$ and $B$ are separable $C^{*}$-algebras. If $A\oplus B$
is WSP1 then both $A$ and $B$ are WSP1.
\end{thm}

\begin{proof}
The proof of the forward direction of Theorem~\ref{thm:sumWSP_iff_eachWSP}
works here just as well.
\end{proof}

\begin{thm}
If $A$ is unital then $A$ is WSP1 if, and only if, $A$ WSP.
\end{thm}

\begin{proof}
When $A$ is unital, $\widetilde{A} \cong A\oplus\mathbb{C}$, so
by Theorem~\ref{thm:WSP1_iff_tilde_A_isWSP}
\[
A\mbox{ is WSP1 }\iff A\oplus\mathbb{C}\mbox{ is WSP}.
\]
By Theorem~\ref{thm:sumWSP_iff_eachWSP}
\[
A\oplus\mathbb{C}
\mbox{ is WSP}
\iff A\mbox{ is WSP}
\mbox{ and }
\mathbb{C}
\mbox{ is WSP}.
\]
We are done, since $\mathbb{C}$ is famously SP and so WSP.
\end{proof}

\begin{thm}
\label{thm:WSP1=00003D=00003D>AANR}
Suppose $X$ is a locally compact, metrizable space.
\begin{enumerate}
\item
If $C_{0}(X)$ is WSP1 then $\alpha X$ is an AANR. 
\item
If $\alpha X$ is an AANR then $C_{0}(X)$ is WSP1 
within the commutative category. 
\end{enumerate}
\end{thm}

\begin{proof}
This follows from Proposition~\ref{pro:AANRviaApproxExten} 
and Definition~\ref{def:WSP1} by essentially the argument used for
Theorem~\ref{thm:WSP=00003D=00003D>PAANR}.
\end{proof}

\begin{example}
Consider $x_{0}$ in the topologist's sine curve, as in 
Figure~\ref{fig:topologist's-sine-curve}(b).  Then
$A=C_{0}(X\setminus\{x_{0}\})$ is WSP1 since $\widetilde{A}\cong C(X)$.
However $\left(A\oplus A\right)^{\sim}\cong C(Y)$ where $Y$ is the
space in Figure~\ref{fig:two-sine-curves-Nasty} and $Y$ is not
an AANR, so $A\oplus A$ is not WSP1.
\end{example}

\bibliographystyle{plain}
\bibliography{WSPrefs}

\end{document}